\documentclass[11pt]{amsart}
\usepackage{graphicx}
\usepackage{amssymb}
\usepackage{epstopdf}
\newtheorem{theorem}{Theorem}

\newtheorem{definition}{Definition}

\newtheorem{proposition}{Proposition}
\newcommand{\diam}{\mbox{diam }}

\title[Doubling Spaces and the Benjamini-Schramm Lemma]{Doubling Metric Spaces are Characterized by a Lemma of Benjamini and Schramm}
\author[J. T.  Gill]{James T.  Gill}
\thanks{The author is supported by an NSF Mathematical Sciences Postdoctoral Research Fellowship DMS-1004721}
\address{Department of Mathematics and Computer Science,
Saint Louis University,
220 N. Grand Blvd.,
St. Louis, MO 63103}
\email{jgill5@slu.edu}
\date{October 31, 2012}
\begin{document}

\begin{abstract}
A useful property of $\mathbb{R}^n$ originally shown by I. Benjamini and O. Schramm  turns out to characterize doubling metric spaces
\end{abstract}


\maketitle

\section{Introduction}

In \cite{BS} I. Benjamini and O. Schramm prove that the distributional limit of rooted random unbiased finite planar graphs with uniformly bounded degree is almost surely recurrent.  By rooted they meant that each graph had a distinguished vertex, and by unbiased they meant that (essentially) given an unrooted graph each vertex is equally likely to be the root.  The distributional limit is the weak limit of the random graphs with respect to the topology of local graph isomorphisms in neighborhoods of the root.  In this influential paper, they prove a curious lemma which we repeat as Lemma A below.

We use the common notation that in metric space $(X,d)$, the ball $B(x,r)$ consists of all points $y \in X$ such that $d(x,y) < r$.  When $K>0$ is positive, we use $KB$ to denote a ball of the same center as $B$, but with radius $K$ times the original.

\begin{definition} Let $w \in C$ be a finite set of points of the metric space $(X,d)$.  The {\em isolation radius} of $w$ is 
\[ \rho_w = \min \{ d(w,v) : v \in C \setminus{\{w\}} \}. \]
The point $w$ is said to be {\em $(\delta, s)$-supported} for some $\delta \in (0,1)$ and $s \geq 2$ if
\[ \inf_{p \in X} \left| \left\{ \left( B(w, \rho_w / \delta) \setminus B(p, \delta \rho_w) \right) \cap C \right\} \right| \geq s. \]
\end{definition}

\vspace{2mm}\noindent{\bf Lemma A.} (Benjamini and Schramm, \cite{BS} Lemma 2.3)
\textit{For every $\delta \in (0,1)$ there is a constant $c = c(\delta)$ such that for every finite $C \subset \mathbb{R}^2$ and every $s \geq 2$ the set of $(\delta, s)$-supported points in $C$ has cardinality at most $c|C|/s$.}\\

Their proof is valid for any $\mathbb{R}^n$ and uses both probability and scaling properties of $\mathbb{R}^n$.  This lemma is a key step in the proof of their recurrence result mentioned above.  It is also known that for $\mathbb{R}^n$ that $c(\delta)$ is of the order $\delta^{-n} \log(1/\delta)$ (see Lemma 3.4 of \cite{GGN}) .  The current author and S. Rohde used this lemma in the proof that Brownian motion on the uniform infinite planar triangulation is recurrent \cite{GR}.  In \cite{BC} I. Benjamini and N. Curien use this same lemma as part of the proof a result similar to that in \cite{BS} but in higher dimensions and more related to sphere packing than to random walks.  In \cite{NPS} H. Namazi, P. Pankka, and J. Souto found the lemma useful to show a certain non-degeneracy of the distributional limit of random Riemannian manifolds with a curvature bound.  Finally, in \cite{GGN} O. Gurel-Gurevich and A. Nachmias proved that the random walk on uniform infinite planar triangulation is recurrent with this now, almost ``magical'' lemma playing a key role.

The usefulness of Lemma A together with its strange opacity leads one to wonder in what settings besides $\mathbb{R}^n$ it might hold.  This thought naturally produces the following definition.

\begin{definition}
A metric space $(X,d)$ is said to {\em carry a Benjamini-Schramm lemma} if for all $\delta \in (0,1)$ there exists a constant $c = c(\delta)$ so that for every finite set $C$ in $X$ and every $s \geq 2$,  the number of $(\delta, s)$-supported points in $C$ is less than $c \cdot |C|/s$ (i.e. the proportion of points which are $(\delta, s)-supported$ is $c/s$ where the constant depends only on $delta$ and of course the underlying metric space).
\end{definition}

One might also think to allow the decay to depend on other functions of $s$, but the proof of Proposition \ref{conv} below shows this is unnecessary.  As Lemma A fails spectacularly in an infinite dimensional normed linear space, some concept of finite dimensionality of our space is needed.   The following definition is a form of finite dimensionality for metric spaces.

\begin{definition}
A metric space $(X,d)$ is called {\em doubling} if there is a constant $D \geq 1$ so that every set of diameter $d$ in $X$ can be covered by at most $D$ sets of diameter at most $d/2$.
\end{definition}
We note that there are many equivalent definitions of doubling (metric) spaces.  In particular, in this note we will force the arbitrary covering sets to be metric balls.  One may also choose any number $l > 1$ and demand that the covering balls be of diameter at most $d/l$.  Euclidean space is clearly doubling.  Hyperbolic space, however, is not.  A non-trivial example of a doubling metric space is the Heisenberg group endowed with the Carnot metric.  It is metric doubling, but not bi-Lipschitz equivalent to $\mathbb{R}^n$.  These particulars are discussed in the book \cite{H} (esp. Chapter 10), and more extensively in the paper \cite{L}.  A metric space which is doubling is often called a space of {\em finite Assouad dimension}.  It turns out that this is precisely the assumption needed for a Benjamini-Schramm lemma.

\begin{theorem} A metric space $(X,d)$ carries a Benjamini-Schramm lemma if and only if $(X,d)$ is doubling. \label{mt}
\end{theorem}

\textit{Acknowledgements} The author would like to thank Steffen Rohde for encouraging work in this direction and the enthusiastic audience of the Saint Louis University analysis seminar for their questions and comments, especially Bryan Clair.

\section{Proof of Equivalence}
We prove Theorem \ref{mt} through a series of propositions. The following proposition is essentially part of the proof of Lemma 5.1 in \cite{NPS}.
\begin{proposition} \label{bleq}
If $(X,d)$ carries a bi-Lipschitz embedding into $\mathbb{R}^n$ then $(X,d)$ carries a Benjamini-Schramm lemma.
\end{proposition}

\begin{proof}
Let us denote the bi-Lipschitz embedding by $f$ and suppose it has Lipschitz constant $L \geq 1$. Suppose that $w \in C$ is a $(\delta, s)$-supported point in $(X,d)$, that is
\[ \inf_{p \in X} \left| \left\{ \left( B(w, \rho_w / \delta) \setminus B(p, \delta \rho_w) \right) \cap C \right\} \right| \geq s. \]
With $w' = f(w)$ and $\rho_{w'}$ denoting the isolation radius of $f(w)$ with respect to $f(C)$ in $\mathbb{R}^n$ with the Euclidean metric. Note that $ \rho_w / L \leq \rho_{w'} \leq L \rho_w$.  So 
\[  f(B(w, \rho_w / \delta)) \subseteq B(w', L^2 \rho_{w'} / \delta) \]
and for any $p \in X$,
\[ f(B(p, \delta \rho_w)) \supseteq B(p, \delta \rho_{w'} / L^2). \]
So $w'$ is $(\delta/L^2, s)$-supported in $\mathbb{R}^n$.  As $\mathbb{R}^n$ carries a Benjamini-Schramm lemma with constant $c(\delta)$, it follows that $X$ carries a Benjamini-Schramm lemma with constant $c(\delta/L^2)$.
\end{proof}

Suppose $(X,d)$ is a metric space and $0 < \epsilon < 1$.  Then we may consider a new distance on $X$ by taking $d^\epsilon (x, y) := (d(x,y))^\epsilon$.  This new metric space $(X, d^\epsilon)$, often called a {\em snowflaked version}
of $(X,d)$, is not in general bi-Lipschitz equivalent to the original space.  Note that if $(X,d)$ has any rectifiable curves, then $(X,d^\epsilon)$ does not.  However, this new metric is still useful due to the following theorem.

\vspace{2mm}\noindent{\bf Theorem B.} (P. Assouad \cite{A})
\textit{Let $(X,d)$ be a doubling metric space and $0 < \epsilon < 1$.  Then $(X, d^\epsilon)$ admits a bi-Lipschitz embedding into $\mathbb{R}^n$, quantitatively.}\\

By quantitatively, we mean that the choice of $n$ and the Lipschitz constant of the embeddding both depend on $\epsilon$, however a refinement of this theorem due to A. Naor and O. Neimann remarkably removed the dependence of $n$ on  $\epsilon$ \cite{NN}.

Noting that Theorem B and Proposition \ref{bleq} imply that every doubling metric space can be snowflaked to become a space which carries a Benjamini-Schramm lemma, the sufficiency part of Theorem \ref{mt} will immediately follow from the proposition below.

\begin{proposition} \label{snbsl}
Suppose for a metric space $(X,d)$ and an $0<\epsilon < 1$ that $(X, d^\epsilon)$ carries a Benjamini-Schramm lemma.  Then $(X,d)$ carries a Benjamini-Schramm lemma.
\end{proposition}

\begin{proof}
Let $C$ be a finite set in $(X,d)$ and $\rho_{w,\epsilon}$ be the isolation radius of $w \in C$ with respect to the metric $d^\epsilon$.  Then
\[ \rho_{w,\epsilon} = \rho_w \hspace{.5mm}^\epsilon. \]
Letting $B^\epsilon$ stand for balls in the $d^\epsilon$ metric and $B$ for the $d$ metric, once we see that
\begin{eqnarray*}
 B^\epsilon (w, \rho_{w,\epsilon} / \delta) & = & \{ x : d^\epsilon(x,w) <\rho_{w,\epsilon} / \delta \}\\
& = & \{x : d(x,w) < (\rho_{w,\epsilon})^{1/\epsilon}/ \delta^{1/\epsilon} \}\\
& = &B(w, \rho_w / \delta^{1/\epsilon})
\end{eqnarray*}
and similarly for any $p \in X$
\[ B^\epsilon (p, \rho_{w,\epsilon} \cdot \delta) = B(p, \rho_w \cdot \delta^{1/\epsilon}) \]
and the proposition follows as in the proof of Proposition \ref{bleq}.
\end{proof}

Now we address the necessity in Theorem \ref{mt}.

\begin{proposition} \label{conv}
Suppose that the metric space $(X,d)$ is not doubling, then the space does not carry a Benjamini-Schramm lemma.
\end{proposition}

\begin{proof}
Let $(X,d)$ be a metric space which is not doubling.  Hence for each $N \in \mathbb{N}$ there exists an $A_N \subset X$ where $\diam A_N = d_n$ needs at least $N$ balls of radius $d_N/2$ to cover $A$.  Consider a cover $\mathcal{G}$ of $A_N$ by balls of radius $d_N/10$.  By a basic covering theorem (see \cite{H} Theorem 1.2) there is a subcover of $\mathcal{G}$, denoted $\mathcal{G'}$, which is disjoint and $5\mathcal{G'}$ covers $A_N$.  By the assumption on $A_N$, the number of balls in $\mathcal{G'}$ is at least $N$.  Let $x_k$ be the centers of the balls in $\mathcal{G}'$.  By the disjointness of $\mathcal{G}'$, $d(x_j, x_k) \geq d_N/5$ if $j \not= k$.  We now set $C$ to be centers of the balls of $\mathcal{G}'$. Let $\delta = 1/20$.  For each $x_j$
\[ B(x_j , \rho_{x_j}/\delta) \supseteq B(x_j, (d_N/5)/\delta) = B(x_j, 4d_n) \]
We claim that $B(x_j, \rho_{x_j}/\delta)$ contains all of the centers $x_k$ of the balls from the $\mathcal{G}'$.  We may assume that for $B \in \mathcal{G}'$, the ball $5B$ intersects the set $A_N$ (as $\mathcal{G}'$ was chosen so the $5B$ cover $A_N$).  As the $\diam A_N = d_N$, a ball of radius $4 d_N$ centered at $x_j$ contains  all the  points $x_k$ (actually only a ball radius of $2 d_N$ is needed).  Now consider the quantity $\rho_{x_j}  \delta$.  As the points $x_j$ are of mutual pairwise distance at least $d_N/5$ apart, if $\rho_{x_j}  \delta$ is less than $d_N/10$, then for any $p \in X$ the ball $B(p, \rho_{x_j}\delta)$ may only contain one of the $x_k$'s.    But $\rho_{x_j} \cdot \delta$ is less than $d_N / 10 $ if and only if
\[ \rho_{x_j} < 2d_N. \]
This is true because, as already noted, the ball $B(x_j, 2 d_N)$ contains all the points $x_k$. Hence each point of the set $C$ consisting of the centers of the balls of $\mathcal{G}'$ is $(1/20, N-1)$-supported.  No metric space carrying a Benjamini-Schramm lemma may contain such sets for each $N \in \mathbb{N}$.
\end{proof}

Taken together, Propositions \ref{bleq}, \ref{snbsl}, and \ref{conv} prove Theorem \ref{mt}.


 \end{document}